\documentclass{article}

\usepackage{amsmath,amssymb,amscd,amsthm,url}

\newtheorem{thm}{Theorem}[section]
\newtheorem{lem}[thm]{Lemma}
\newtheorem{cor}[thm]{Corollary}

\theoremstyle{remark}
\newtheorem{rem}[thm]{Remark}

\theoremstyle{definition}

\newcommand{\isomto}{\overset{\sim}{\rightarrow}}

\title{Effective convergence bounds for Frobenius structures on connections}
\author{Kiran S. Kedlaya \and Jan Tuitman}

\begin{document}
\maketitle

\begin{abstract}
Consider a meromorphic connection on $\mathbb{P}^1$ over a $p$-adic field. In many cases, such as those
arising from Picard-Fuchs equations or Gauss-Manin connections, this connection admits a Frobenius structure
defined over a suitable rigid analytic subspace. We give an effective convergence bound for this Frobenius
structure by studying the effect of changing the Frobenius lift. We also give an example indicating 
that our bound is optimal.
\end{abstract}

\section{Introduction}

In recent years, much work has gone into using $p$-adic cohomology as an effective
tool for numerical computation of zeta functions (and some related quantities) of algebraic varieties 
over finite fields. One important technique in this field is the \emph{deformation method}
of Lauder, in which one computes the zeta function of a variety by fitting it into a one-parameter
family of varieties, constructing the associated Picard-Fuchs equation or Gauss-Manin connection,
then exploiting the existence of a Frobenius structure on this differential equation to reduce the problem
to another member of the family (the initial condition, so to speak). Another important technique
is the \emph{fibration method}, also introduced by Lauder, in which one uses similar techniques to compute 
the zeta function of the total space of a one-parameter family, again starting from a single fiber. 
See \cite{laudeform,laufibration} for further discussion.

To execute the deformation and fibration methods in practice, it is necessary to have not just the existence
of a Frobenius structure, but explicit bounds on its convergence within a given residue disc. 
Concretely, these bounds are needed to enable the reconstruction of a rational function from a power series
expansion, by bounding the degrees of its zero and pole divisors.
One can often obtain crude bounds by direct calculations, but it is essential to have more accurate bounds in 
order to limit the required intermediate precision needed to achieve a final result to a given level of accuracy.

A technique for obtaining accurate bounds has been  suggested by the first author in the preprint
\cite{kedmainz}, under the assumptions (satisfied in many cases in practice) that the differential equation
has at most one singularity in any residue disc, and that the exponents of the local
monodromy at such a singularity are $p$-adically integral. The idea is to exploit the parallel
transport of Frobenius structures between two choices of a Frobenius lift, to reduce the question of
convergence within a given residue disc to the same question with the Frobenius lift centered around the
singularity in the disc, a question which can be solved rather easily.

However, the bound given in \cite[Theorem~6.5.10]{kedmainz} is not best possible. A stronger bound was claimed 
in the original (2008) manuscript of \cite{kedmainz}, but the second author discovered that the proof 
was incomplete, as it relied on some unjustified assertions about the convergence of solutions of $p$-adic 
differential equations. In this paper, we give a corrected version of the original argument,
thus giving a stronger version of \cite[Theorem~6.5.10]{kedmainz}. We also provide a numerical example 
which indicate that the resulting bound is sharp.

\section{The theorem}

We first introduce some notation and terminology.

Let $p$ denote a prime, $n$ a positive integer, and $\mathbb{F}_q$ the finite field with $q=p^n$ elements.
We write $\mathbb{Q}_q$ for the unique unramified extension of degree $n$ of the field of $p$-adic numbers
$\mathbb{Q}_p$, and $\mathbb{Z}_q$ for the ring of integers of $\mathbb{Q}_q$. 
Let $U$ be an open dense subscheme of $\mathbb{P}^1_{\mathbb{Q}_q}$ with nonempty
complement $Z$. Suppose that $\mathcal{E}$ is a vector bundle on $U$ equipped with 
a connection $\nabla$, and let $t$ denote some coordinate on $\mathbb{P}^1_{\mathbb{Q}_q}$. 

We write
$\sigma$ for
the standard $p$-th power Frobenius lift on $\mathbb{P}^1_{\mathbb{Q}_q}$ that is, the (semilinear) map 
that lifts the $p$-th power Frobenius map on $\mathbb{P}^1_{\mathbb{F}_q}$ and
satisfies $\sigma(t)=t^p$. 

Let $V$ denote the rigid analytic subspace of $\mathbb{P}^1_{\mathbb{Q}_q}$ which is the complement
of the union of the open disks of radius $1$ around the points of $Z$, and 
$\mathcal{O}^{\dagger}(U)$ 
the ring of functions that converge on some strict neighbourhood of $V$. A \textit{Frobenius structure} 
on $\mathcal{E}$ with respect to $\sigma$ is an isomorphism $\mathcal{F}:\sigma^* \mathcal{E} \isomto \mathcal{E}$ 
of vector bundles with connection defined on some strict neighbourhood of $V$. 

We fix a basis $[v_1,\ldots,v_r]$ of sections of $\mathcal{E}$ on $U$, define matrices $N \in M_r(\mathcal{O}(U))$ and 
$\Phi \in M_r(\mathcal{O}^{\dag}(U))$, such that
\begin{align*}
\nabla v_j &= \sum_{i=1}^r N_{ij} v_i \otimes dt, \\
\mathcal{F} v_j &= \sum_{i=1}^r \Phi_{ij} v_i,
\end{align*}
and call these the matrices of $\nabla$ and $\mathcal{F}$. Note, however, that $\nabla$ and $\mathcal{F}$ are not 
$\mathcal{O}(U)$- and $\mathcal{O}^{\dag}(U)$-linear, respectively. Instead, $\nabla$ satisfies the Leibniz rule, and $\mathcal{F}$ is
$\sigma$-semilinear as a map from $\mathcal{E}$ to itself.

Since $\mathcal{F}$ is a morphism of vector bundles with connection, it is horizontal with respect to $\nabla$. This implies
that the matrices $N$ and $\Phi$ satisfy the differential equation
\begin{align} \label{frobeq}
N\Phi+\frac{d\Phi}{dt}= \frac{d\sigma(t)}{dt}\Phi \sigma(N) = pt^{p-1} \Phi \sigma(N).
\end{align}

Now let $z$ be a geometric point of $Z$, and suppose that the entries of $N$ have at most a simple pole at $z$. When $\nabla$ is
a Gauss-Manin connection, by the regularity theorem we can always choose the basis $[v_1,\ldots,v_r]$ so that this is the case 
(where the choice will in general depend on $z$). The \textit{exponents} of $\nabla$ at $z$ with respect to $[v_1,\ldots,v_r]$ 
are defined as the eigenvalues of the matrix $(t-z)N|_{t=z}$. When $\nabla$ is a Gauss-Manin connection or admits a Frobenius 
structure, these are known to be rational numbers. 

Let $|.|$ denote the norm on $\mathcal{O}^{\dag}(U)$ induced by the supremum norm on $V$, and $v_p(.)$ the corresponding
discrete valuation, so that $|.|=p^{-v_p(.)}$. Extend both of these to $M_r(\mathcal{O}^{\dag}(U))$ in the usual way, 
i.e as the maximum and minimum over the entries, respectively.

\begin{thm} \label{frobexp}

Let $z$ be an unramified geometric point of $Z$, and assume that $Z$ does not contain any 
other points with the same reduction modulo $p$. Suppose that $[v_1,\ldots,v_r]$ 
is a basis of $\mathcal{E}$ with respect to which the matrix $N$ of $\nabla$ has 
at most a simple pole at $z$, and the exponents $\{\lambda_1,\ldots,\lambda_r\}$ 
of $\nabla$ at $z$
are contained in $\mathbb{Q} \cap \mathbb{Z}_p$. Assume that $\mathcal{E}$ 
admits a Frobenius structure $\mathcal{F}$ with respect to $\sigma$, and let $\Phi$ be the 
matrix of $\mathcal{F}$ with respect to the basis $[v_1,\ldots,v_r]$. For $i \in \mathbb{N}$,
put
\begin{align*}
f(i) &= \max\{(v_p(\Phi)+v_p(\Phi^{-1})) \lceil \log_p(i) \rceil, (r-1)v_p(N)+(v_p(\Phi)+v_p(\Phi^{-1}))\lfloor \log_p(i) \rfloor \},
\end{align*}
and define 
\begin{align*}
c = \begin{cases}
0 & \mbox{if $v_p(N) \geq 0$} \\
\min\{0, i + f(i): i \in \mathbb{N}\} & \mbox{if $v_p(N) < 0$}.
\end{cases}
\end{align*}
For $m \in \mathbb{N}$, put
\begin{align*}
g(m) &= \max \{i \in \mathbb{N} \; | \; i+ v_p(\Phi) + c + f(i)  < m \},
\end{align*}
and define
\begin{align*}
\alpha_1 &= \lfloor -p \min_i \{ \lambda_i \} + \max_{i} \{\lambda_i\} \rfloor, \\ 
\alpha_2 &=  \left \{ 
         \begin{array}{cl}
         0  & \mbox{if $N$ does not have a pole at $z$},  \\
         0  & \mbox{if $z \in \{0,\infty \}$}, \\
         g(m) & \mbox{otherwise}.
         \end{array}
         \right. 
\end{align*}
Then $\Phi$ is congruent modulo $p^{m}$ to a matrix of rational functions of order greater 
than or equal to $-(\alpha_1+p \alpha_2)$ at $z$ (that is, the entries of the difference between 
the two matrices all have $p$-adic valuation at least $m$). 
\end{thm}

The proof proceeds in several steps. We start with the following lemma.

\begin{lem} \label{semstab}
Let $N=\sum_{i=-1}^\infty N_i t^i$ be an $r \times r$ matrix such that $tN$ converges on the
open unit disk and $N_{-1}$ is a nilpotent matrix. Let $\Phi=\sum_{i=-\infty}^{\infty} \Phi_i t^i$ 
be an $r \times r$ matrix that converges on some open annulus of outer radius $1$. Suppose that
$N,\Phi$ satisfy equation (\ref{frobeq}). Then $\Phi_i=0$ for all $i<0$, so that $\Phi$ converges
on the whole open unit disk.
\end{lem}

\begin{proof}
See \cite[Proposition $17.5.1$]{kedbook}.
\end{proof}

When the exponents of $N$ at $0$ are not necessarily zero, this can be generalized as follows.

\begin{lem} \label{semstabcor}
Let $N=\sum_{i=-1}^\infty N_i t^i$ be an $r \times r$ matrix such that $tN$ converges on the
open unit disk and the eigenvalues $\lambda_1,\ldots,\lambda_r$ of $N_{-1}$ are rational numbers 
with denominators coprime to $p$. Let $\Phi=\sum_{i=-\infty}^{\infty} \Phi_i t^i$ be an $r \times r$ 
matrix that converges on some open annulus of outer radius $1$. Suppose that $N,\Phi$ satisfy 
equation (\ref{frobeq}). Then $\Phi_i=0$ whenever
\begin{align*}
i < p \min_j \{\lambda_j \} - \max_j \{\lambda_j \}.
\end{align*}
\end{lem}

\begin{proof}
First we may adjoin $t^{1/k}$ for $k$ coprime to $p$ (if necessary), to reduce to the case where
$\lambda_1,\ldots,\lambda_r \in \mathbb{Z}$. In that case, by applying so called \emph{shearing transformations}, 
one can find an invertible $r \times r$ matrix $W$ over $\mathbb{Q}_q(t)$ such that the matrix
\begin{align*}
N'=W^{-1} N W +  W^{-1} \frac{dW}{dt}
\end{align*}
still has (at most) a simple pole at $t=0$, but now with all exponents equal to $0$. Moreover, one 
can ensure that $t^b W$ and $t^{-a} W^{-1}$ do not have a pole at $t=0$, for $a=\min_j \{ \lambda_j\}$
and $b=\max_j \{ \lambda_j \}$. More details on this can be found in \cite[Lemma~5.1.6]{kedmainz}. If we 
change basis to the basis given by the colums of $W$, then
\begin{align*}
N &\rightarrow N', \\
\Phi &\rightarrow \Phi' =  W^{-1} \Phi \sigma(W).
\end{align*}  
Now Lemma \ref{semstab} can be applied to the pair $N',\Phi'$, so that $\Phi'_i=0$ for all $i<0$.
Since $\Phi=  W \Phi' \sigma(W^{-1})$, this implies that $\Phi_i=0$ for all 
$i <p a - b$. 
\end{proof}

Recall that we have chosen the standard $p$-th power Frobenius lift $\sigma$. However, we could just as 
well have chosen a different lift. The following lemma allows one to change from one Frobenius lift to another.

\begin{lem} \label{transf}
Let $\mathcal{D}$ denote the differential operator on $\mathcal{E}$ defined by $\nabla v = \mathcal{D} v \otimes dt$. 
Suppose that $\mathcal{E}$ admits a Frobenius structure $\mathcal{F}_1:\sigma_1^* \mathcal{E} \isomto \mathcal{E}$ with respect
to a Frobenius lift $\sigma_1$, and let $\sigma_2$ be some other Frobenius lift. Then $\mathcal{E}$ also admits a 
Frobenius structure $\mathcal{F}_2:\sigma_2^* \mathcal{E} \isomto \mathcal{E}$ with respect to $\sigma_2$, defined by
\begin{align*}
\mathcal{F}_2(v)=\sum_{i=0}^{\infty} (\sigma_2(t)-\sigma_1(t))^i \mathcal{F}_1\left(\frac{\mathcal{D}^i}{i!}(v)\right).
\end{align*}
\end{lem}

\begin{proof}
See \cite[Proposition $17.3.1$]{kedbook}.
\end{proof}

Finally, we need a bound on the matrices of the differential operators $\frac{\mathcal{D}^i}{i!}$ that appear in Lemma \ref{transf}.

\begin{lem} \label{lemDi/i!}
Let $\Delta^{(i)}$ be the matrix of the differential operator $\frac{\mathcal{D}^i}{i!}$ with respect to the basis $[v_1,\ldots,v_r]$ that is,
\[
\left( \frac{\mathcal{D}^i}{i!} \right) v_k = \sum_{j=1}^r \Delta^{(i)}_{jk} v_j.
\]
Then we have
\begin{align*}
v_p(\Delta^{(i)}) \geq f(i),
\end{align*}
where $f(i)$ is defined as in Theorem \ref{frobexp}.
\end{lem}

\begin{proof}
Let $\eta$ denote a generic point of the disk of radius $1$ around $z$. One can verify that the Taylor series
\begin{align*}
T(-t+\eta,v_j)= \sum_{i=0}^{\infty} (-t+\eta)^i \frac{\mathcal{D}^i(v_j)}{i!}
\end{align*}
defines a horizontal section of $\nabla$ that meets $v_j$ at the point $\eta$. Form the matrix $M$ whose $j$-th column 
consists of the expression of $T(-t+\eta,v_j)$ in terms of the basis $[v_1,\dots,v_r]$, then expand 
$M = \sum_{i=0}^\infty M_i (t - \eta)^i$. Since $\nabla$ does not have any singularities in the open disk of radius 
$1$ around $\eta$, it follows from \cite[Theorem 18.3.3]{kedbook} that  
\begin{align*}
\min\{v_p(M_0), \dots, v_p(M_i)\} &\geq (v_p(\Phi)+v_p(\Phi^{-1})) \lceil \log_p(i) \rceil, \\
\intertext{and from \cite[Remark 18.3.4]{kedbook} (with $q=p$) that}
\min\{v_p(M_0), \dots, v_p(M_i)\} &\geq (r-1)v_p(N) + (v_p(\Phi)+v_p(\Phi^{-1})) \lfloor \log_p(i) \rfloor. 
\end{align*}
Since $|\Delta^{(i)}|$ attains its maximum at $\eta$, and
\[
M_i = (-1)^i \Delta^{(i)} (\eta) +  \left( \mbox{ terms coming from } \Delta^{(j)} \mbox{ with } j<i \right),
\]
we deduce the bound by induction on $i$.
\end{proof}

Now we finally get to the proof of Theorem \ref{frobexp}.

\begin{proof}[Proof of Theorem \ref{frobexp}] 
We first note that in case $z=0$ or $z= \infty$, the claim is clear from Lemma \ref{semstabcor}.

Suppose next that $z$ is a point at which $N$ has no pole, so that $\alpha_1 = \alpha_2 = 0$.
If we use the Frobenius lift $\sigma'$ with $\sigma'(t-z)=(t-z)^p$, then by Lemma \ref{semstabcor} again
(applied after translating $z$ to the origin),
the Frobenius matrix $\Phi'$ with respect to $\sigma'$ is holomorphic at $z$.
By Lemma \ref{transf} (with $\sigma_1=\sigma', \sigma_2=\sigma$), $\Phi$ is also holomorphic at $z$,
proving the claim in this case.

Finally, suppose that $N$ does have a pole at $z$.
In this case, Lemma \ref{semstabcor} implies that $\Phi'$ has order at least $-\alpha_1$ at $z$. 
We may again use Lemma \ref{transf} (with $\sigma_1=\sigma', \sigma_2=\sigma$) 
to convert back to the original Frobenius lift; this gives us the identity
\[
\Phi = \sum_{i=0}^\infty p^i u^i \Phi' \sigma'(\Delta^{(i)}),
\]
where $pu = (t-z)^p + \sigma(z) - t^p$ (with $v_p(u) \geq 0$), and $\Delta^{(i)}$ again denotes the matrix of the differential 
operator $\frac{\mathcal{D}^i}{i!}$ with respect to the basis $[v_1,\ldots,v_r]$.
In this identity, the summand at index $i$ has order at least $-\alpha_1-pi$ at $z$, 
and $p$-adic valuation at least $i + v_p(\Phi') + f(i)$ 
by Lemma \ref{lemDi/i!}. This will give the desired bound once we check that $v_p(\Phi') \geq v_p(\Phi)+c$. 
To see this, apply Lemma \ref{transf} with $\sigma_1$ and $\sigma_2$ interchanged to obtain
\[
\Phi' = \sum_{i=0}^\infty \frac{p^i}{i!} (-u)^i \Phi \sigma(i! \Delta^{(i)}).
\]
If $v_p(N) < 0$, we get the claim by invoking Lemma \ref{lemDi/i!} again; if $v_p(N) \geq 0$,
we instead note that  $v_p(p^i/i!)$ and $v_p(i! \Delta^{(i)})$ are both nonnegative.
\end{proof}

The following corollary is often useful when the matrix $N$ of $\nabla$ with
respect to some basis does not have a simple pole at $z$.

\begin{cor}
Suppose that $[v_1,\ldots,v_r]$ is a basis for $\mathcal{E}$ as in Theorem \ref{frobexp}, and let  $[w_1,\ldots,w_r]$ 
be another basis for $\mathcal{E}$, such that $v_j = \sum_{i=1}^r W_{ij} w_j$ with $W \in M_r(\mathbb{Q}_q(t))$. Then  
the matrix $\Phi'$ of $\mathcal{F}$ with respect to $[w_1,\ldots,w_r]$ is congruent modulo $p^{m+v_p(W)+v_p(W^{-1})}$ to a matrix 
of rational functions of order greater than or equal to 
\[
-(\alpha_1+p \alpha_2(m))+\mathrm{ord}_z(W) + p \; \mathrm{ord}_z(W^{-1})
\]
at z.
\end{cor}

\begin{proof}
The matrix $\Phi'$ satisfies
\[
\Phi'=W \Phi \sigma(W)^{-1}.
\]
\end{proof}

\begin{rem} \label{improvement}
In some special cases Theorem \ref{frobexp} can still be improved a little.
\begin{enumerate} 
\item If $\sigma(z)=z^p$ (such a $z$ is called a \emph{Teichm\"uller lift}), then $\sigma(t)-\sigma'(t)$ is
divisible by $t-z$ in the proof of Theorem \ref{frobexp}. So when we apply Lemma \ref{transf}, 
some cancellation occurs, and modulo $p^m$ the matrix $\Phi$ has order greater than or equal to
$-(\alpha_1+(p-1)\alpha_2(m))$ at $z$.

\item  Suppose that $z \neq 0,\infty$. If we denote the residue matrix $(t-z)N|_{t=z}$ of $N$ at $z$ by $R_z$, and the identity
matrix by $I$, then the leading term in the Laurent series expansion of the matrix $\Delta^{(i)}$ of
$\frac{\mathcal{D}^i}{i!}$ at $(t-z)$ is given by
\[
(R_z-(i-1)I) \ldots (R_z-I) R_z \; \frac{(t-z)^{-i}}{i!}.
\]
In many cases 
there exists $S \in GL_r(\mathbb{Q}_q)$ such that $S^{-1} R_z S$ is diagonal. Writing $\lambda_1,\ldots,\lambda_r$ for the 
entries on the diagonal, the leading term can then be written as
\[
S 
\begin{pmatrix}
\frac{\lambda_1(\lambda_1-1) \ldots (\lambda_1-(i-1))}{i!} &  & 0 \\
 &  \ddots &  \\
0 &  & \frac{\lambda_r(\lambda_r-1) \ldots (\lambda_r-(i-1))}{i!}
\end{pmatrix}
S^{-1}(t-z)^{-i}. 
\]
However, since $\lambda_1,\ldots,\lambda_r \in \mathbb{Q} \cap \mathbb{Z}_p$ by assumption, 
the matrix in the middle is easily seen to have entries in $\mathbb{Q} \cap \mathbb{Z}_p$ as
well. 
This means that the valuation of the leading term in the the Laurent series expansion
of $\Delta^{(i)}$ at $z$ is bounded by $v_p(S)+v_p(S^{-1})$. Now when we apply Lemma
\ref{transf}, we see that if $g(m)+v_p(\Phi)+c+v_p(S)+v_p(S^{-1}) \geq m$, then modulo 
$p^m$ the matrix $\Phi$ has order greater than or equal to $-(\alpha_1+p(\alpha_2(m)-1))$ at $z$. 
This is related to the improvement upon \cite[Theorem~18.2.1]{kedbook} given by the theorem of
Dwork and Robba on which it is based \cite{dwork-robba}.
\end{enumerate}
\end{rem}

\section{An example: a family of elliptic curves}

We consider the family given by the affine equation
\begin{align*}
y^2 = x^3+1+(t+1)(x^2+x),
\end{align*}
The closure of the zero locus of this equation in 
$\mathbb{P}^2_{\mathbb{Q}} \times \mathbb{P}_{\mathbb{Q}}^1$ 
defines a family $X/U$ of elliptic curves over 
$U=\mathbb{P}^1_{\mathbb{Q}}-\{-2,2\}$.

The relative algebraic de Rham cohomology $H^1_{dR}(X/U)$ is a vector bundle on $U$, 
and carries a natural Gauss-Manin connection $\nabla$. Moreover, $H^1_{dR}(X/U)$ is of rank $2$, 
and a basis is given by
\[
\left[ \frac{dx}{y},\frac{xdx}{y} \right].
\]
Let $p$ be an odd prime number. The space $H^1_{dR}(X/U)$ equipped with $\nabla$ coincides with the relative rigid 
cohomology $H^1_{rig}(X_p/U_p)$ of the reduction $X_p/U_p$ of $X/U$ modulo $p$, and therefore it admits a Frobenius 
structure $\mathcal{F}$ with respect to the standard lift $\sigma$ of the $p$-th power Frobenius. Let $N$ and $\Phi$
be the matrices of $\nabla$ and $\mathcal{F}$ with respect to the above basis, respectively. We compute

\[
N = \frac{1}{t^2-4}
\begin{pmatrix} 
-\frac{t}{2}-\frac{1}{2} & \frac{t}{2}+\frac{3}{2} \\
-\frac{1}{2} & \frac{t}{2}+\frac{1}{2}
\end{pmatrix}. 
\]

It is known that $v_p(\Phi)=0$, and $v_p(\Phi^{-1})=-1$ in this case, and clearly $v_p(N)=0$, 
so in Theorem \ref{frobexp} we have
\[
g(m)=\max \{i \in \mathbb{N} \; | \; i - \lfloor \log_p(i) \rfloor  < m \}.
\]

\subsection{$z=2$}

At $z=2$ the exponents are $\{-1/4,1/4\}$. The residue matrix $R_2$ is diagonalizable by 
integral matrices for $p \neq 5$. So by remark \ref{improvement}, for $p \neq 5$ the bound 
from Theorem \ref{frobexp} for the order of $\Phi$ modulo $p^m$ can be improved to 
$-(\lfloor \frac{p+1}{4} \rfloor+p(g(m)-1))$, while for $p=5$ it remains $-(1+5g(m))$. \\

Experimentally, we find that for $p=3$ the order is bounded by $1-3(m-1)$, for $p=5$ it is 
bounded by $1-5(m-1)$, and for $p=7$ it is bounded by $2-7(m-1)$, all for $m$ up to $250$ and 
with equality for many $m$. 

\subsection{$z=-2$}

At $z=-2$ the exponents are $\{0,0\}$, so the bound from Theorem \ref{frobexp} for the order of 
$\Phi$ modulo $p^m$ is given by $-pg(m)$. \\

Experimentally, we find that for $p=3$ the bound is sharp for $m=1,2,3,6,8,17,\\25,52,78,159,239$, 
for $p=5$ for $m=1,2,3,4,5,10,15,20,24,49,74,99,123,248$, and for 
$p=7$ for $m=1,2,3,4,5,6,7,14,21,28,35,42,48,97,146,195,244$, all for $m$ up to $250$.
 
\section*{Acknowledgments}

Kedlaya was supported by NSF (CAREER grant DMS-0545904, grant DMS-1101343), 
DARPA (grant HR0011-09-1-0048),
MIT (NEC Fund, Cecil and Ida Green professorship), 
and UCSD (Stefan E. Warschawski professorship).
Tuitman was supported by the European Research Council (grant 204083).

\bibliography{effective}

\end{document}